\documentclass[11pt]{article}
\usepackage{indentfirst,latexsym,bm}
\usepackage{amsfonts}
\usepackage{amssymb}
\usepackage[leqno]{amsmath}
\usepackage{amsbsy}
\usepackage{dsfont}
\usepackage{leftidx}
\usepackage{amsthm}
\usepackage{amscd}
\usepackage[all]{xy}
\usepackage{chemarrow}
\usepackage{multirow}
\usepackage{hyperref}
\usepackage{color,xcolor}
\begin{document}
\title{Categorical Morita equivalence and monoidal Morita equivalence of semisimple Hopf algebras of dimension $pqr$}
\author{Zhiqiang Yu\thanks{Email:\,zhiqyu-math@hotmail.com}
\\{\small Department of Mathematics,} \\ {\small East China Normal University,
Shanghai 200241, China}
}
\date{}
\maketitle

\newtheorem{question}{Question}
\newtheorem{defi}{Definition}[section]
\newtheorem{conj}{Conjecture}
\newtheorem{thm}[defi]{Theorem}
\newtheorem{lem}[defi]{Lemma}
\newtheorem{pro}[defi]{Proposition}
\newtheorem{cor}[defi]{Corollary}
\newtheorem{rmk}[defi]{Remark}
\newtheorem{Example}{Example}[section]

\newcommand{\C}{\mathcal{C}}
\newcommand{\D}{\mathcal{D}}
\newcommand{\M}{\mathcal{M}}
\newcommand{\N}{\mathcal{N}}

\theoremstyle{plain}
\newcounter{maint}
\renewcommand{\themaint}{\Alph{maint}}
\newtheorem{mainthm}[maint]{Theorem}

\theoremstyle{plain}
\newtheorem*{proofthma}{Proof of Theorem A}
\newtheorem*{proofthmb}{Proof of Theorem B}
\abstract
In this paper, we determine the cocycle deformations and Galois objects for semisimple Hopf algebras of dimension $pqr$, where $p,q,r$ are distinct primes, and decide the categorically Morita equivalent classes and monoidally Morita equivalent classes of them. We show that all of them only have one trivial Galois objects, therefore these Hopf algebras are pairwise twist inequivalent, equivalently they are not monoidally Morita equivalent to each other, moreover, all the categorically Morita equivalent classes are determined.

\section{Introduction}
Throughout this paper, we always assume $k$ is an algebraically closed and has characteristic zero, $k^{*}:=k\backslash {\{0}\}$, $\mathbb{Z}_r:=\mathbb{Z}/r$.

For any finite abelian category $\mathcal{A}$, we use $rank(\mathcal{A})$ to denote the cardinal of set of non-isomorphic simple objects of $\mathcal{A}$.

A fiber functor is a faithful exact tensor functor from a finite tensor category $\mathcal{C}$ to the tensor category $Vec$ \cite{EGNO}, where $Vec$ is the category of finite-dimensional vector spaces.
Ulbrich \cite{Ul} proved that for a finite-dimensional Hopf algebra $H$, the right Galois objects of $H$ are in bijection with the fiber functors on $\mathcal{M}^H\cong Rep(H^{*})$, where $\mathcal{M}^H$ and $Rep(H)$ are the finite dimensional  $H$-comodules category and $H$-modules category respectively.

For finite tensor category $\mathcal{C}$, a left module category $\mathcal{M}$ is equivalent to a tensor functor $\Psi:\mathcal{C}\rightarrow End(\mathcal{M})$, $X\mapsto \Psi_X$, where $End(\mathcal{M})$ is the (multi-)tensor category of functors from $\mathcal{M}$ to $\mathcal{M}$, $\Psi_X(M):=X\otimes M$, $\forall X\in \mathcal{C}, M\in \mathcal{M}$. For a fiber functor $F:\mathcal{C}\rightarrow Vec$, we can regard the abelian category $Vec$ as a $\mathcal{C}$-module category via the fiber functor $F$, that is $X\otimes V:=F(X)\otimes V$, $V\in Vec$. Note that $rank(Vec)=1$.

Conversely, for a rank one $\mathcal{C}$-module category $\mathcal{M}$, we have an abelian category equivalence $Vec\cong \mathcal{M}$. By the axiom of module category, we see there is a unique tensor functor $F:\mathcal{C}\rightarrow End(\mathcal{M})\cong Vec$, that is a fiber functor $F$. Henceforth, there is a bijection between fiber functors on $\mathcal{C}$ and rank one module categories over $\mathcal{C}$.

Schauenburg \cite{S1} showed that for finite-dimensional Hopf algebras $H, L$, their finite-dimensional representation categories are tensor equivalent iff they are differed by a Drifeld twist, that is $H\cong L^J$ as Hopf algebras. Therefore, there is another bijective correspondence between Drinfeld twists of $H$ and fiber functors on $Rep(H)$.

To determine the Galois objects or fiber functors of an arbitrary finite-dimensional Hopf algebra seems impossible nowadays, since we need to classify all finite-dimensional Hopf algebras first; moreover, there does exist finite-dimensional Hopf algebra $H$ which has infinite many non-isomorphic Drinfeld twists, i.e. there are infinitely many inequivalent fiber functors on $Rep(H)$ \cite{EG3}. While, for fusion category $\mathcal{C}$, there are finitely many inequivalent $\mathcal{C}$-module categories (particularly the finiteness hold for fiber functors)\cite{ENO1}; hence, we turn to the semisimple Hopf algebras, whose representation categories are fusion categories. And \c{S}tefen\cite{St}  also proved that for any semisimple Hopf algebra $H$, there are finitely many fiber functors on $Rep(H)$.

Ostrik \cite{O1} showed that for a finite tensor category $\mathcal{C}$, any exact indecomposable module category $\mathcal{M}$ is equivalent to the representation category $Mod_\mathcal{C}(A)$ of an algebra $A\in \mathcal{C}$. In general, it is difficult to construct all these algebras. However, for the group-theoretical fusion category $\mathcal{C}(G,\omega, F,\beta)$, they can be constructed more explicitly, Ostrik proved  all module categories can be represented by using pair $(L,\alpha)$ \cite{O2}. But the classification of equivalent classes is incomplete, the correct criteria is given by Natale \cite{Na2}.

Based on the classifications of semsimple Hopf algebras  and fusion categories  of dimension $pqr$, where $p, q, r$ are distinct primes, and $r<q<p$. All these semisimple Hopf algebras are fitting into some  abelian extensions, therefore they are group-theoretical \cite{Na1,Na3,ENO2}. In this article, we first determine the Galois objects of semisimple Hopf algebras of dimension $pqr$, following the references \cite{CMNVW} and \cite{XY}, and then consider two kinds of equivalence of finite-dimensional Hopf algebras: the monoidal Morita equivalence and the categorical Morita equivalence.

This is our first main result on the monoidal Morita equivalence:
\begin{mainthm}[Theorem \ref{monoidalM}] Semisimple Hopf algebras of dimension $pqr$ have exactly one Galois object, so they are not monoidally Morita equivalent to each other or Drinfeld twist equivalent to each other.
\end{mainthm}
Next, we turn to the categorically Morita equivalent classes of them. Based on the classification of semisimple  Hopf algebras of dimension $pqr$ \cite{Na3,ENO2}, and classification of finite groups of order $pqr$ \cite{H,A}, we obtain the following second main  theorems:

First of all, we determine the cases where only exist trivial Hopf algebras, i.e. under the assumption $rq\nmid(p-1)$.
\begin{mainthm}[Proposition \ref{onesemiproduct} and \ref{case4}]If $rq\nmid (p-1)$, the number of categorically Morita equivalent classes semisimple Hopf algebras of dimension $pqr$ are the number of non-isomorphic groups of order $pqr$ for cases $(1)-(3)$ and $(5)-(6)$, and for case $(4)$ the number is $4$ $($if $r=2)$ or $3+\frac{r-1}{2}$ $($if $r>2)$, where the cases $(i)$ is denoting the cases of groups in [Table \ref{table}].
\end{mainthm}

\begin{mainthm}[Theorem \ref{categoricalM1}, \ref{categoricalM2}]
\begin{description}
\item[$(a)$] If $rq\mid (p-1)$, $r\nmid (q-1)$, then there are four categorically inequivalent classes : $k[\mathbb{Z}_{pqr}]$,  $k[(\mathbb{Z}_{p}\rtimes\mathbb{Z}_{r})\times \mathbb{Z}_{q}]$, $k[(\mathbb{Z}_{p}\rtimes\mathbb{Z}_{q})\times \mathbb{Z}_{r}]$ and $\mathcal{A}_p(q;r)$, where the right actions $\rtimes$ are non-trivial.
\item[$(b)$] If $rq\mid (p-1)$,  $r\mid (q-1)$, then there are $6$ $($if $r=2)$ or $5+\frac{r-1}{2}$ $($if $r>2)$ categorically inequivalent classes.
\end{description}
\end{mainthm}

The paper is organized as follow. In section \ref{Preliminaries}, we first recall some basic definitions and properties that we will used throughout, like group-theoretical fusion categories, cocycle deformations, metacyclic groups, etc. In section \ref{Semsimplepqr}, we will summarize the classification of  semisimple Hopf algebras of dimension $pqr$. In section \ref{monMorita}, we compute the Galois objects of all semisimple Hopf algebras of dimension $pqr$, and determine their monoidally Morita equivalent classes. In section \ref{catMorita}, we determine the categorically Morita equivalent classes of semisimple Hopf algebras of dimension $pqr$.
\section{Preliminaries} \label{Preliminaries}
In this section,  we will recall some notations like pointed fusion categories, group-theoretical fusion categories and their module categories, and Galois objects, abelian extensions, metacyclic groups, etc. We refer \cite{O1,O2,ENO1,ENO2,EGNO} for the basic definitions and properties of fusion categories, module categories, Drinfeld centers; for Galois objects and $2$-cocycle deformations see \cite{S1}, for abelian extensions see \cite{K,Na1}, for metacyclic groups and their Schur multipliers see \cite{A, B, Ro}.
\subsection{Group-theoretical fusion categories}
For a  fusion category $\mathcal{C}$, abelian category $\mathcal{M}$ is a left $\mathcal{C}$-module category means that there is  a tensor functor $\Psi:\mathcal{C}\rightarrow End(\mathcal{M})$, $X\mapsto \Psi_X$, where $\Psi_X(M):=X\otimes M$, $\forall X\in \mathcal{C}, M\in \mathcal{M}$. $\mathcal{M}$ is an indecomposable module category if it can not be decomposed as the direct of two non-trivial module subcategories. We denote $\mathcal{C}^*_\mathcal{M}:=Fun_\mathcal{C}(\mathcal{M},\mathcal{M})$, the set of $\mathcal{C}$-module functors from $\mathcal{M}$ to $\mathcal{M}$.

A fusion categories is pointed if all the simple objects are invertible. Any pointed fusion category
$\mathcal{C}$ is equivalent to a finite group $G$-graded vector space $Vec^\omega_G$, where $\omega\in H^3(G,k^{*})$ is a $3$-cocycle, with the association isomorphism $(g\otimes h)\otimes l\mapsto g\otimes(h\otimes l) $ given by the cocycle $\omega(g,h,l)$.

\begin{defi}Two fusion categories $\mathcal{C}, \mathcal{D}$ are categorical Morita equivalent if there exists an semisimple indecomposable  $\mathcal{C}$-module category $\mathcal{M}$ s.t $\mathcal{C}^*_{\mathcal{M}}\cong \mathcal{D}^{op}$ as fusion categories, where $\D^{op}$ is the opposite fusion category of $\D$.
\end{defi}
\begin{defi}\label{catgoricalHopfalgebra}Two semisimple Hopf algebras $H$, $L$ are categorically Morita equivalent if their finite-dimensional representation categories $Rep(H)$ and $Rep(L)$ are categorically Morita equivalent.
\end{defi}
\begin{thm}[\cite{ENO2}]\label{Drinfeldcenter}Fusion categories $\mathcal{C}$ and  $\mathcal{D}$ are categorically Morita equivalent via $\mathcal{C}$-module category $\mathcal{M}$ iff their Drinfeld centers $Z(\mathcal{C})$, $Z(\mathcal{D})$ are braided equivalent.
\end{thm}
\begin{rmk}Indeed, Theorem \ref{Drinfeldcenter} is true for any finite tensor category $\mathcal{C}$, see \cite{DN,EGNO} for details.
\end{rmk}

A fusion category $\mathcal{D}$ is $\textit{group-theoretical}$ if it is categorically Morita equivalent to a pointed fusion category $Vec^\omega_G$, that is, there exists indecomposable semisimple $\mathcal{D}$-module category $\mathcal{M}$ s.t $\mathcal{D}^{op}_\mathcal{M}\cong Vec^\omega_G$.

Ostrik  \cite{O2} proved that for pointed  fusion category $Vec^\omega_G$, every exact indecomposable module category $\mathcal{M}$ is equivalent to the representation category of an algebra $A=k_\beta[L]\in Vec^\omega_G$, where $L\subseteq G$ is a subgroup and $\beta:L\times L\rightarrow k^*$ is a $2$-cochain satisfying $d^2\alpha=\omega_{\mid F\times F\times F}$. And we denote the module category that is determined by pair $(L,\beta)$ as $\mathcal{M}(L,\beta)$, the group-theoretical fusion category $(Vec^\omega_G)^*_{\mathcal{M}(L,\beta)}$ is denoted by $\mathcal{C}(G,\omega,L,\beta)$.

Let us denote $Mod(\mathcal{C})$ the $2$-category of left exact $\mathcal{C}$-module categories. It is shown that there a $2$-category equivalence  $Mod(\mathcal{C})\cong Mod((\mathcal{C}^{*}_\mathcal{M})^{op})$, $\mathcal{N}\mapsto Fun_\mathcal{C}(\mathcal{M},\mathcal{N})$, for any $\mathcal{C}$-module category $\mathcal{N}$ \cite{EGNO}, so every semsimple indecomposable module category of the group-theoretical fusion category is also determined \cite{O2}.

A 2-cocycle $\alpha\in H^2(G,k^*)$ is non-degenerate if the twisted group algebra $k_\alpha[G]$ is a simple matrix algebra. As said in the introduction, a fiber functor on $\mathcal{C}$ is nothing but a rank one $\mathcal{C}$-module category.

\begin{thm}[\cite{O2},\cite{Na2}] \label{fiberfunctor}Fiber functors on $\mathcal{C}(G, \omega, F,\alpha)$
correspond to pairs $(L,\beta)$, where $L$ is a subgroup of $G$ and $\beta$ is a $2$-cocycle on $L$,
such that the following conditions are satisfied:
\begin{enumerate}
  \item The class of $\omega|_{ L\times L\times L}$ is trivial;
  \item $G = LF$; and
  \item The class of the $2$-cocycle $\alpha|_{F\cap L}\beta^{-1}|_{L\cap F}$ is non-degenerate.
\end{enumerate}
Two fiber functors $(L,\beta),(L{'},\beta{'})$ are isomorphic iff
there exists an element $g\in G$ such that $L{'}=gLg^{-1}$, and the cohomology class of
the two cocycle $\beta{'}^{-1}\beta^g\Omega_g$ is trivial in $H^2(L, k^{*})$, where $\beta^g(h,l):=\beta(ghg^{-1},glg^{-1})$ for $h,l\in L$ and $\Omega_g(a,b):=\frac{\omega(gag^{-1},gbg^{-1},g)\omega(g,a,b)}{\omega(gag^{-1},g,b)}$ for $a, b\in G$.
\end{thm}

\subsection{Monoidal Morita equivalence}
\begin{defi}Two finite-dimensional Hopf algebras $H, L$ are monoidally Morita equivalent if their finite-dimensional comodule categories $\mathcal{M}^H$ and $\mathcal{M}^L$ are tensor equivalent.
\end{defi}
\begin{defi}For Hopf algebra $H$, a right $H$-comodule algebra $A$ is a Galois object if the Galois morphism $\beta:A\otimes A\rightarrow A\otimes H$ is an isomorphism, and $A^{co H}=k$, where $\beta(a\otimes b)=ab_{(0)}\otimes b_{(1)}$, $\forall a, b\in A$.
\end{defi}
For a right $H$-Galois object $A$, we can define a new Hopf algebra $L:=(A^{op}\otimes A)^{co H}$, which is unique up to isomorphism of Hopf algebras, see \cite{S1} for details.\begin{defi}For a Hopf algebra $H$, a Drinfeld twist $J\in H\otimes H$ is an invertible element satisfies $$(\Delta\otimes id_H)(J)(J\otimes 1)=(id_H\otimes \Delta) (J)(1\otimes J),\quad (\varepsilon\otimes id_H)J=(id_H\otimes\varepsilon)J=1.$$
\end{defi}
\begin{defi}For Hopf algebra $H$, a cocycle $\sigma:H\otimes H\rightarrow k$ is a convolution invertible morphism satisfies: $\forall g, h, l\in H$
\begin{gather*}
\sigma(1,h)=\varepsilon(h)=\sigma(h,1),\\
 \sigma(g_{(1)},h_{(1)})\sigma(g_{(2)}h_{(2)},l)=\sigma(h_{(1)},l_{(1)})\sigma(g,h_{(2)}l_{(2)}).
\end{gather*}
\end{defi}
 For Drinfeld twist $J$ and cocycle $\sigma$, we can define new Hopf algebras $H^J$ and $H^\sigma$ by changing their comultiplication and multiplication respectively, called twist deformation and $2-$cocycle deformation, see \cite{Ra} for the explicit constructions and the duality between them.
\begin{thm}[\cite{S1}]\label{Drinfeldtwist}Two finite-dimensional Hopf algebras $H, L$ are monoidally Morita equivalent iff there exists an $(L,H)$-biGalois object iff there exists a cocycle $\sigma$ satisfies $H^\sigma\cong L$ as Hopf algebras. Equivalently, $Rep(H^*)\cong Rep(L^*)$ iff $H^*\cong (L^*)^J$ for some Drinfeld twist.
\end{thm}
Therefore, by Theorem \ref{Drinfeldcenter} and Theorem \ref{Drinfeldtwist}, we have
\begin{cor}\label{Drinfelddouble}Semisimple Hopf algebras $H, L$ are categorically Morita equivalent iff $Rep(D(H))\cong Rep(D(L))$ as braided fusion categories iff $D(H)\cong D(L)^J$ as Hopf algebras, $J$ is a Drinfeld twist.
\end{cor}
Combining the conclusions of \cite{S1} with \cite{Ul}, we have:
\begin{pro} \label{Oneonecorrespondence}
For finite-dimensional Hopf algebra H, there is a bijective correspondence between the following sets:
\begin{enumerate}
  \item The set of right Galois objects of $H^{*}$;
  \item The set of fiber functors on $Rep(H)$;
  \item The set of Drinfeld twists on H;
  \item The set of 2-cocycles of $H^*$.
\end{enumerate}
\end{pro}

\subsection{Abelian extensions}
Let $\Gamma,~F$ be finite groups, we say they  form a matched pair $(F,\Gamma,\lhd,\rhd)$, if there are actions $\lhd: F\times\Gamma\rightarrow F$ and $\rhd: \Gamma\times F \rightarrow \Gamma$
satisfying for all $x,~y\in F,~s,~t\in \Gamma$, we have
$$s\rhd (xy) = (s \rhd x)((s \lhd x) \rhd y), \quad (st)\lhd x = (s \lhd (t \rhd x))(t \lhd x).$$

For a matched pair $(F,\Gamma,\lhd,\rhd)$, the Cartesian product $F\times \Gamma$ admits a group structure, with product given by:
\begin{gather*}
(x, s)(y, t)=(x(s\rhd y), (s\lhd y)t).
\end{gather*}
 This group denoted by $F\bowtie\Gamma$, and called the $\textit{bicrossed product}$ of $F,\Gamma$.

Moreover, we have an abelian extension $k\rightarrow k^{\Gamma}\rightarrow k^{\Gamma}\leftidx{^\tau}{\#_{\sigma}k[F]}\rightarrow k[F]\rightarrow k$, with product and coproduct given by
\begin{gather*}
(e_g\#x)(e_h\#y) =\delta_{g\lhd x,h} \sigma_g(x, y)e_g\#xy,\\
\Delta(e_g\#x) =\sum_{st=g}\tau_x(s,t)e_s\#(t \rhd x)\otimes e_t\#x.
\end{gather*}
where $\sigma=\sum\limits_{g\in \Gamma}\sigma_g\delta_g :F\times F\rightarrow (k^\Gamma)^{\times}$ is a normalized $2$-cocycle , $\delta_g$ is the orthogonal primitive idempotent satisfying $\langle\delta_g,h\rangle=\delta_{g,h}$, $\tau=\sum\limits_{s\in \Gamma}\sigma_x\delta_x :\Gamma\times\Gamma\rightarrow (k^F)^{\times}$ is a normalized $2$-cocycle, $\sigma_g(x,y):=\sigma(x,y)(g)$, $\tau_x(s,t):=\tau(s,t)(x)$. And the cocycles $\tau, \sigma$  subject to the compatible conditions : $\forall s,t\in \Gamma, x,y\in F$,
\begin{gather*}
\sigma_1(s,t)=1,\qquad\tau_1(x,y)=1,\\
\sigma_{ts}(x,y)\tau_{xy}(t,s)=\tau(t,s)\tau(t\lhd(s\rhd x),s\lhd x)\sigma_{t}(s\rhd x,(s\lhd x)\rhd y)\sigma_{s}(x,y).
\end{gather*}

We call an abelian extension $k\rightarrow k^{\Gamma}\rightarrow k^{\Gamma}\leftidx{^\tau}{\#_{\sigma}k[F]}\rightarrow k[F]\rightarrow k$ is splitting, if $\tau=1$ and $\sigma=1$.

Given a matched pair $(F,\Gamma,\lhd,\rhd)$, the set of equivalent classes of extensions $k\rightarrow k^{\Gamma}\rightarrow k^{\Gamma}\leftidx{^\tau}{\#_{\sigma}k[F]}\rightarrow k[F]\rightarrow k$ giving rise to these actions is denoted by $Opext(k^\Gamma,k[F])$, see \cite{Na1,K} for more details.

By a result of Kac \cite{K}, there is an exact sequence
\begin{align*}
& 0\rightarrow H^1(F\bowtie\Gamma,k^{*})\xrightarrow{res}H^1(F,k^{*})\oplus H^1(\Gamma,k^{*})\\
& \rightarrow Aut(k^{\Gamma}\leftidx{^\tau}{\sharp_{\sigma}k[F]})\rightarrow H^2(F\bowtie\Gamma,k^{*})\xrightarrow{res}H^2(F,k^{*})\oplus H^2(\Gamma,k^{*})\\
& \rightarrow Opext(k^\Gamma,k[F])\xrightarrow{\overline{\omega}}H^3(F\bowtie\Gamma,k^{*})\xrightarrow{res}
H^3(F,k^{*})\oplus H^3(\Gamma,k^{*})\rightarrow\cdots.
\end{align*}

And the element $[\tau,\sigma]\in Opext(k^\Gamma,k[F])$ is mapped under $\overline{\omega}$ onto 3-cocycle $\omega(\tau,\sigma)\in H^3(F\bowtie\Gamma,k^{*})$, which is defined by: for $\forall a,b,c \in F\bowtie\Gamma$
$$\omega(\tau,\sigma)(a,b,c):=\tau_{\pi(c)}(p(a)\lhd \pi(b),p(b))\sigma_{p(a)}(\pi(b),p(b)\rhd\pi(c)),$$
where $\pi: F\bowtie\Gamma\rightarrow F, (g,h)\mapsto g$, and $p: F\bowtie\Gamma\rightarrow\Gamma, (g,h)\mapsto h$ are projections.
\begin{thm}[\cite{Na1}]
Let $(F,\Gamma,\lhd,\rhd)$ be a matched pair of finite groups. Suppose that $H$ is a Hopf algebra fitting into an abelian extension
 $$k\rightarrow k^{\Gamma}\rightarrow H\rightarrow k[F]\rightarrow k$$
 associated to  $(F,\Gamma,\lhd,\rhd)$. Then  $H$ is group-theoretical, and  $\text{Rep}\,(H)\cong \mathcal{C}(F\bowtie\Gamma,\omega(\tau,\sigma),F,1)$ as fusion categories.
\end{thm}

\subsection{Schur multiplier of metacyclic groups}

\begin{defi}A finite group G is a metacyclic group if it contains a cyclic normal subgroup $H$ s.t the quotient group $G/H$ is cyclic.
\end{defi}
\begin{Example}Cyclic groups are metacyclic groups; direct products or semi-products cyclic groups are metacyclic groups, such as  the dihedral groups $D_{2n}$, $n\in \mathbb{Z},~n>2$.
\end{Example}
\begin{thm}[\cite{Ro}]For a finite group G, it is a metacyclic group iff all its Sylow subgroups are cyclic.
\end{thm}
Therefore, by Lagrangian theorem of finite groups, we see any finite group of order square-free is a metacyclic group.

The following proposition gives the explicit group structure of metacyclic groups.
\begin{pro}[\cite{Ro}]Every finite metacyclic group is isomorphic to groups like:
$\langle a,b\mid a^m=1, bab^{-1}=a^r, b^n=a^{{\frac{m\lambda}{(m,r-1)}}}\rangle$, where $m,~n$ are natural numbers, and $r,~\lambda$ are integers, and $r^n\equiv1 (mod ~m)$. We denote it  by $G(m,n,r,\lambda)$. And two metacyclic groups $G(m,n,r,\lambda)\cong G(m,n,r,\lambda{'})$, where $\lambda{'}=(\lambda,h(m,n,r))$ is the maximal common divisor of $\lambda$ and $h(m,n,r)$, and $h(m,n,r):=\frac{(m,1+r+\cdots+r^{n-1})(m,r-1)}{m}$.
\end{pro}
The Schur multiplier $H^2(G,k^*)$ of metacyclic group $G$ had been computed already, here we only give the conclusion and omit its proof.
\begin{pro}[\cite{B}]
For metacyclic group $G(m,n,r,\lambda)$, the Schur multiplier  $H^2(G(m,n,r,\lambda), k^{*})$ is a cyclic group of order $(\lambda,h(m,n,r))$.
\end{pro}
By taking $\lambda=1$, we get a special case, which we mainly use later:
\begin{cor}\label{trivialSchur}For any finite metacyclic group $G:=\langle a,b\mid a^m=1,b^n=a^{\frac{m}{(m,r-1)}}, bab^{-1}=a^r,\rangle$ with $r^n\equiv 1 (mod~m)$,~$(m,n)=(n,r-1)=1$, we have $H^2(G,k^{*})=0$ .
\end{cor}
Note that a finite group of $pq$ has trivial Schur multiplier group, where $p,q$ are distinct primes.
\section{Semisimple Hopf algebras of dimension $pqr$}\label{Semsimplepqr}
From now on, we always assume $r<q<p$ be three distinct primes.

In this section, we make a summary of the classification of semisimple Hopf algebras of dimension $pqr$, see references \cite{ENO2,Na3}.

Recall that a fusion category $\mathcal{C}$ is integral, if for any simple object $X\in\mathcal{C}$, the Frobenius-Perron dimension $FPdim(X)$ is a positive integer, see \cite{ENO1,EGNO} for a systematic discussions of Frobenius-Perron dimension of fusion categories and module categories.
\begin{thm}[\cite{ENO2}]Integral fusion categories of FP-dimension $pqr$ are group-theoretical, and semisimple Hopf algebras of dimension $pqr$ are obtained by an splitting abelian extension, that is, it is isomorphic to a smash product $k^\Gamma\# k[F]$ for exact factorization $G=F\Gamma$, where $G$ is a finite group of order $pqr$.
\end{thm}

Therefore, all semisimple Hopf algebras of dimension $pqr$ are isomorphic to one of the following forms $k^G$, $k[G]$, $k^\Gamma\# k[F]$($\Gamma$ and $F$ are subgroups of orders less than $pqr$).

Next proposition gives a characterization on primes $p, q, r$,  when there there will exist a non-trivial Hopf algebra of dimension $pqr$.
\begin{pro}[\cite{Na3}]\label{DualHopfalgebras}Let H be non-trivial semisimple Hopf algebra of dimension pqr, and  is obtained by an abelian extension $k\rightarrow k^\Gamma\rightarrow H\rightarrow k[F]\rightarrow k$, then $rq\mid (p-1)$, and $\Gamma$ is the only non-abelian metacyclic group of order pq $($or pr$)$, $F=\mathbb{Z}_r(or~\mathbb{Z}_q)$, G=$F\Gamma$ is an exact factorization of group of order pqr.
\end{pro}

To the end of this section, let us assume $\Gamma:=\langle b,c\mid b^p=a^r=1,~aba^{-1}=b^{t}\rangle$, where $r\mid (p-1)$,~$t^r\equiv1 ~(mod~p)$, that is $\overline{t}$ is a $r$-th primitive root of unity in finite field $\mathbb{Z}_p$, where $\overline{t}:=t(mod~p)$.
\begin{pro}[\cite{Na3}]\label{fusioncategory}Let H be non-trivial semisimple Hopf algebra of dimension pqr, then H is isomorphic $\mathcal{A}_p(r; q)$ or $(\mathcal{A}_p(r; q))^{*}\cong \mathcal{A}_p(q; r)$ as a Hopf algebra, where semisimple  Hopf algebra $\mathcal{A}_p(r; q):=k^\Gamma\#k[\mathbb{Z}_q]$ associated to the action by group automorphisms $\Gamma\lhd \mathbb{Z}_q$ given by
$$b \lhd g=b^m,\quad a \lhd g = a,\quad b \rhd g= a \rhd g = g $$
 where g is a generator of $\mathbb{Z}_q$, and $m^q\equiv1 ~(mod~p)$.
\end{pro}
It is obviously  group $G=F\bowtie\Gamma\cong \mathbb{Z}_p\rtimes \mathbb{Z}_{qr}$, where the action $\rtimes$  is the conjugate action, and the Hopf algebra structures of $\mathcal{A}_p(r; q)=k^\Gamma\#k[\mathbb{Z}_q]$ is given by following:
\begin{gather*}
(e_g\#x)(e_h\#y) =\delta_{g\lhd x,h}e_g\#xy,\quad \Delta(e_g\#x) =\sum_{st=g} e_s\# x \otimes e_t\#x.
\end{gather*}
\section{Monoidally Morita equivalent classes}\label{monMorita}

We first consider the Galois objects of these semisimple Hopf algebras of dimension $pqr$ ($p,q,r$ are three distinct primes, $r<q<p$), and then to find the cocycle deformations(or Drinfeld twists) on them.

First of all, we consider the trivial Hopf algebras, i.e. Hopf algebras that are isomorphic to  group algebras  or their dual Hopf algebras.

In fact, there is  a more general result for all finite-groups of square-free.
\begin{pro}\label{group}For group algebras $k[G]$ or dual Hopf algebras $k^G$, where $G$ a is suqare-free finite group, they only have one trivial Galois objects.
\end{pro}
\begin{proof}As explained in section \ref{Preliminaries}, Galois objects on $k^G$ are in one-to-one correspondence with the fiber functors $Rep(G)$ by proposition \ref{Oneonecorrespondence}. While Fiber functors on $Rep(G)$ are term like $\mathcal{M}(L,\alpha)$, where $L\subseteq G$ is a subgroup, $\alpha\in H^2(S,k^{*})$ is a non-degenerate $2$-cocycle, so that  $\mathcal{M}(L,\alpha)\cong Vec$ by Theorem \ref{fiberfunctor}. As $G$ is square-free, then such subgroup only can be trivial group, i.e. there is only trivial Galois object.

For Galois objects of $k[G]$, they are in one-to-one correspondence with the fiber functors on fusion category $Rep(k^G)\cong \mathcal{M}^{k[G]}=Vec_G$. And as $\mathcal{C}^*_\mathcal{C}\cong\mathcal{C}^{op}$ \cite{EGNO}, whence $Vec_G=\mathcal{C}(G,1,G,1)$. Using Theorem \ref{fiberfunctor}, which shows fiber functors in turn are in one-to-one correspondence with conjugacy classes of pairs $(S,\alpha)$, where  $S\subseteq G$ is a subgroup, $\alpha\in H^2(S,k^{*})$ is a non-degenerate $2$-cocycle. As $G$ is square-free, therefore there only have one trivial Galois object on $k[G]$.
\end{proof}
\begin{cor}\label{trivial}If $H$ is a semisimple Hopf algebra, and $Rep(H)\cong Rep(G)$, where G is of square-free order. Then $H\cong k[G]$. Dually, if $\mathcal{M}^H\cong \mathcal{M}^{{k^G}}$, then $H\cong k^G$.
\end{cor}
\begin{proof}If $Rep(H)\cong Rep(G)$, then $H\cong k[G]^J$ for some Drinfeld twist by Theorem \ref{Drinfeldtwist} . Therefore, $H^{*}\cong (k^G)^{\sigma}$ for some cocycle $\sigma$. By proposition \ref{group} , we know that the dual group algebra $k^G$ has no  non-trivial Galois objects, so the cocycle deformation is trivial, that is $(k^G)^{\sigma}\cong k^G$ as Hopf algebras, therefore, $H\cong k[G]$ as Hopf algebras. The proof of the second statement is same.
\end{proof}
By the classification of triangular semisimple Hopf algebras of \cite{EG1}, we have
\begin{cor}Any non-trivial semisimple Hopf algebras of dimension square-free can not be triangular.
\end{cor}
In fact, any odd-dimensional square-free braided fusion category is tensor equivalent to  the representation category of finite group, see \cite{BrNa}.
\begin{lem}[\cite{Na3}] \label{NormalSubgroup}Let G be a group of order pqr. Then G has a unique $($normal$)$ subgroup H of order pq.
\end{lem}
\begin{pro}\label{thefirstcase}For semisimple Hopf algebra $\mathcal{A}_p(q;r)$, it only have trivial Galois object.
\end{pro}
\begin{proof}Since the Hopf algebra $\mathcal{A}_p(q;r)$ is obtained by splitting abelian extension $k\rightarrow k^\Gamma\rightarrow \mathcal{A}_p(q;r)\rightarrow k[F]\rightarrow k$, $\Gamma=\mathbb{Z}_p\rtimes \mathbb{Z}_q$, $F=\mathbb{Z}_r$ by proposition \ref{fusioncategory} .Therefore, the representation category $Rep(\mathcal{A}_p(q;r))$ is equivalent to $ \mathcal{C}(G,1,F,1)$ as fusion categories, where $G:=F\bowtie\Gamma=\mathbb{Z}_r\bowtie(\mathbb{Z}_p\rtimes \mathbb{Z}_q)$. Since $F=\mathbb{Z}_r$, by Theorem \ref{fiberfunctor} , any fiber functor on fusion category $\mathcal{C}(G,1,F,1)$ is corresponding to a subgroup $S$ of order $pq$ and a non-degenerate 2-cocycle $\alpha_{\mid \Gamma\cap F}$, where $\alpha\in H^2(\Gamma,k^{*})$. But we know $H^2(\Gamma, k^{*})=0$ by corollary \ref{trivialSchur} and there exists only one subgroup of $G$ of order $pq$ by lemma \ref{NormalSubgroup}, so there is only one trivial Galois object of $\mathcal{A}_p(q;r)$.
\end{proof}
\begin{pro}For Hopf algebra $\mathcal{A}_p(r;p)$, it only has exactly one trivial Galois objects.
\end{pro}
\begin{proof}First of all, by Proposition \ref{fusioncategory} we have fusion categories equivalence $Rep(\mathcal{A}_p(r;p))\cong \mathcal{C}(G,1,F,1)$, where $G:=\mathbb{Z}_q\bowtie(\mathbb{Z}_p\rtimes \mathbb{Z}_r)$, $F=\mathbb{Z}_q$, $\Gamma=\mathbb{Z}_p\rtimes \mathbb{Z}_r$.  As proposition \ref{thefirstcase}, we only need to find all the conjugation classes of subgroups $L \subseteq G$ of order $pr$, since the Schur multiplier $H^2(L,k^{*})=0$ for $L=\mathbb{Z}_p\rtimes \mathbb{Z}_r$ or $\mathbb{Z}_p\times\mathbb{Z}_r$
by corollary \ref{trivialSchur}.

From the action of $\mathbb{Z}_q$, we know the subgroup $\mathbb{Z}_p\rtimes \mathbb{Z}_r$ is a normal subgroup of $G$. In this case, since the action of $\mathbb{Z}_q$ on $\mathbb{Z}_r$ is trivial, the group $G\cong \mathbb{Z}_p\rtimes \mathbb{Z}_{qr}$, using Sylow theorem, we know there is only one normal subgroup of order $p$; for any subgroup $U$ of order $pr$, there must be  an element of the form $b^ic^j\neq 1$  (where $\mathbb{Z}_q=\langle b\rangle$,~$\mathbb{Z}_r=\langle c\rangle$, $i,~j$ are non-negative integers, and $i+j\neq0$), while $bc=cb$, so if $i,~j>0$, this element has order $qr$, therefore  subgroup $U$ has order $pqr$, impossible. Hence, $U=\mathbb{Z}_p\rtimes \mathbb{Z}_r$, that is it only have one normal subgroup of order $qr$, and the Hopf algebra only have trivial Galois object.
\end{proof}

\begin{thm}\label{monoidalM}For any semisimple Hopf algebra of dimension $pqr$, it does not admit any non-trivial Drinfeld twist or non-trivial cocycle deformation. That is, they are  monoidally Morita inequivalent to each other.
\end{thm}
\begin{proof}The statement is proved for group algebras and their dual Hopf algebras in corollary \ref{trivial} . For Hopf algebras $\mathcal{A}_p(q;r)$ and $\mathcal{A}_p(r;q)$, as they has only one trivial Galois object, so it does not have any non-trivial cocycle deformations, and the dual Hopf algebras  do not have any non-trivial Drinfeld twist. And  $(\mathcal{A}_p(q;r))^{*}\cong \mathcal{A}_p(r;q)$ as Hopf algebras, so we finish the proof.
\end{proof}

\section{Categorically Morita equivalent classes}\label{catMorita}
Although semisimple Hopf algebras of dimension $pqr$ are not monoidally Morita equivalent to each other by Theorem \ref{monoidalM}, they might be  categorically Morita equivalent. The next interesting question is how many  categorically Morita equivalent classes for semisimple Hopf algebras of dimension $pqr$.

\begin{defi}Two finite groups G and $G{'}$ are categorically Morita equivalent
 if the graded vector spaces fusion categories $Vec_G$ and $Vec_{{G{'}}}$ are categorically Morita equivalent.
\end{defi}
\begin{defi}[\cite{EG2}]Two finite groups G and $G{'}$ are isocategorical if their representation categories $Rep(G)$ and $Rep(G{'})$ are tensor equivalent, i.e. $k^G$ and $k^{{G{'}}}$ are monoidally Morita equivalent.
\end{defi}
\begin{rmk}There exist examples that two finite groups $G, G{'}$ are categorical Morita equivalent, but $Rep(G)\ncong Rep(G{'})$, therefore $G\ncong G{'}$ \cite{Nai1}.
The smallest order  of groups s.t $G_1$, $G_2$ are categorically Morita equivalent but not isocategorical is $16$, see \cite{Nai2}.
\end{rmk}
By Theorem \ref{monoidalM}, two finite groups of order square-free are isocategorical iff they are isomorphic.

Next, we will recall a criteria of  categorical Morita equivalence between finite groups. Here, we take the notations of \cite{Nai1}. For a subgroup $H\subseteq G$, let $G/H$ be the right coset, $\pi:G\rightarrow G/H$ is the natural projection, $x\mapsto Hx$.
And take ${\{u(x)\in G\mid x\in G/H}\}$ be a representative elements, where $u$ is the section $G/H\rightarrow G$, $p\circ u=id$. The set ${\{u(x)\mid x\in G/H}\}$ is a $G$-set with action $u(x)\lhd g=p(u(x)g)$, $\forall g\in G$. The elements $u(x)g$ and $u(x\lhd g)$ is in the same coset, they are differed by an element $\kappa_{x,g}\in H$, i.e. $u(x)g=\kappa_{x,g}u(x\lhd g)$. It is easy to see the identity $\kappa_{{x,g_1g_2}}=\kappa_{{x,g_1}}\kappa_{{x\lhd g_1,g_2}}$ is true, $\forall x\in G/H, g_1, g_2\in G$.

Moreover, $\kappa$ defines an element in $H^2(G/H,H)$, $\kappa(x_1,x_2)=\kappa_{{x_1,u(x_2)}}$, and corresponds to group extension $1\rightarrow H \rightarrow G\rightarrow G/H\rightarrow 1$ \cite{Nai1}.

In the  paper \cite{Nai1}, the sufficient and necessary condition for graded fusion categories $Vec^\omega_G$ and $Vec^{{\omega{'}}}_{{G{'}}}$ to be categorical Morita equivalent is given. Here, we only list the special case: $\omega=1$, $\omega{'}=1$.

\begin{pro}[\cite{Nai1}]\label{GroupMorita}Two finite groups G and $G{'}
$ are categorical Morita equivalent if and only
if the following conditions hold:
\begin{description}
\item[$(1)$] G contains a normal abelian subgroup H;
\item[$(2)$] there exists a G-invariant $\mu \in H^2(H,k^{*})$ s.t  $G{'}\cong\widehat{H} \rtimes_{\nu}K$ as finite groups; the group structure of $\widehat{H} \rtimes_{\nu}K$ is given by
\begin{align*}(\rho_1,x_1)(\rho_2,x_2)=(\nu(x_1,x_2)\rho^{{x_2}}_1\rho_2,x_1x_2),
\end{align*}
where  $\langle\rho^x,g\rangle:=\langle\rho,xgx^{-1}\rangle$, $\nu=\Phi(\mu)$, $\Phi: H^2(H,k^{*})^K\rightarrow H^2(K,\widehat{H})$, $K=G/H$, see \cite{Nai1} for details.
 \item[$(3)$]  The 3-cocycle $\varpi\in H^3(\widehat{H}\rtimes_\nu (G/H), k^*)$ defined below is trivial.
\begin{align*}\varpi((\rho_1,x_1),(\rho_2,x_2),(\rho_3,x_3)):=
(\widetilde{\nu}(x_1,x_2)(u(x_3)))(1)\rho_1(\kappa(x_2,x_3)),
\end{align*}
where $\widetilde{\nu}:K\times K\rightarrow C^1(G,C)$, $\widetilde{\nu}(y_1, y_2)
=\frac{\leftidx{^{y_2}}{\eta_{{y_1}}}\eta_{{y_2}}}{\eta_{{y_1 y_2}}}$, $u(x) \in G/H$ is a representative of element $x\in G$.
\end{description}
\end{pro}

Next, we  need find all the isomorphism classes of metacyclic groups of order $pqr$. First of all, there is a classical formula for the number of isomorphic classes of square-free groups :
\begin{thm}[\cite{H}] The number f(n) of groups of order n, where n is square-free is given by following equality
\begin{gather*}
f(n)=\sum_{m\mid n}\prod_{{p_{1}}}\frac{p^{c(p_1)}_1-1}{p_1-1}
\end{gather*}
where $p_1$ runs over all prime divisors of $\frac{n}{m}$ and $c(p_1)$ is the number of prime divisors $q_1$ of m that satisfy $q_1\equiv 1 (mod ~p_1)$.
\end{thm}
And in the reference \cite{A}, Alonso gave a complete description of all the types of metacyclic groups of order $pqr$, we list them below:
\begin{align*}
G_1:=\langle a,b,c\mid a^p=b^q=c^r=1, ab=ba,ac=ca,bc=cb\rangle,
\end{align*}
\begin{align*}
G_2:=\langle a,b,c\mid a^p=b^q=c^r=1,ab=ba,ac=ca, cbc^{-1}=b^{K(q,r)}\rangle,
\end{align*}
\begin{align*}
G_3:=\langle a,b,c\mid a^p=b^q=c^r=1,ab=ba,cac^{-1}=a^{K(p,r)},bc=cb\rangle,
\end{align*}
\begin{align*}
G_4:=\langle a,b,c\mid a^p=b^q=c^r=1,ab=ba,cac^{-1}=a^{K(p,r)},cbc^{-1}=b^{k(q,r)}\rangle,
\end{align*}
where $k(q,r)=K(q,r)^n, n=1,\cdots r-1$, $K(q,r)^r\equiv1 (mod~q)$.
\begin{align*}
G_5:=\langle a,b,c\mid a^p=b^q=c^r=1,bab^{-1}=a^{K(p,q)},ac=ca,bc=cb\rangle,
\end{align*}
\begin{align*}
G_6:=\langle a,b,c\mid a^p=b^q=c^r=1, bab^{-1}=a^{K(p,q)}, cac^{-1}=a^{K(p,r)},bc=cb\rangle.
\end{align*}

In summary, the number of groups of order $pqr$ is given in [Table \ref{table} ].

\begin{table}
\centering
\begin{tabular}{|c|c|c|c|c|}
  \hline
   case &$q\mid (p-1)$ & $r\mid (p-1)$ & $r\mid (q-1)$ & Number \\
 \hline (1)&No & No & No & 1 \\
 \hline (2)&No & No & Yes & 2 \\
  \hline (3)& No & Yes & No & 2 \\
 \hline (4)&No & Yes & Yes & $r$+2 \\
 \hline (5)&Yes & No & No & 2 \\
 \hline (6)&Yes & No & Yes & 3 \\
 \hline (7)&Yes & Yes & No & 4 \\
 \hline (8)&Yes & Yes & Yes & $r$+4 \\
  \hline
\end{tabular}
\caption{Number~ of ~groups~ of~ order~ $pqr (r<q<p)$}\label{table}
\end{table}

For later use, we denote
\begin{align*}
\widetilde{G_4}:=\langle a,b,c\mid a^p=b^q=c^r=1,ab=ba,cac^{-1}=a^{\widetilde{K}(p,r)},cbc^{-1}=b^{\widetilde{k}(q,r)}\rangle,
\end{align*}
where $\widetilde{k}(q,r)=\widetilde{K}(q,r)^n, n=1,\cdots r-1$, $\widetilde{K}(q,r)^r\equiv1 (mod~q)$, and $\widetilde{K}(p,r)$ $K(q,r)\equiv1 (mod~q)$, and call it is ``dual" to $G_4$.

\begin{lem}\label{abelian}In any case, Hopf algebra $k[\mathbb{Z}_{pqr}]=k[\mathbb{Z}_{p}\times\mathbb{Z}_{q}\times \mathbb{Z}_{r}]$ is not categorically Morita equivalent to any other Hopf algebras.
\end{lem}
\begin{proof}
Note  there only one abelian group $\mathbb{Z}_{pqr}$, the Hopf algebra $k[\mathbb{Z}_{pqr}] =k[\mathbb{Z}_{p}\times\mathbb{Z}_{q}\times \mathbb{Z}_{r}]\cong k^{\mathbb{Z}_{p}\times\mathbb{Z}_{q}\times \mathbb{Z}_{r}}$.
For any semisimple Hopf algebras $H$ of dimension $pqr$, if $H$ is categorically Morita equivalent to $k[\mathbb{Z}_{pqr}]$ then $D(H)\cong D(k[\mathbb{Z}_{pqr}])^J$ by corollary \ref{Drinfelddouble}. However, the Drinfeld double of $k[\mathbb{Z}_{p}\times\mathbb{Z}_{q}\times \mathbb{Z}_{r}]$ commutative, so $D(H)\cong D(k[\mathbb{Z}_{pqr}])$. We know that $D(H)$ is commutative iff $H, H^*$ are commutative \cite{Ra}, hence $k[\mathbb{Z}_{pqr}] $ can not be  categorically Morita equivalent to any other Hopf algebras, as $\mathbb{Z}_{pqr}$ is the only abelian group of order $pqr$.
\end{proof}
Before we handle other cases, let us determine the morphisms $\nu$, $\widetilde{\nu}$, $\mu$, $\varpi$ in proposition \ref{GroupMorita} . Note that for groups $H$ of order $st$ or $s$ ($s, t$ are distinct primes), we have $H^2(H,k^{*})=0$ by corollary \ref{trivialSchur}, for $\mu \in H^2(H,k^{*})=0$, therefore we have $\nu=1$, $\widetilde{\nu}=1$.

Since $\kappa$ defines an element in $H^2(K,H)$,  and corresponds to extension $1\rightarrow H \rightarrow G\rightarrow G/H\rightarrow 1$. Therefore, if the exact sequence $1\rightarrow H \rightarrow G\rightarrow G/H\rightarrow 1$ is splitting, that is $G\cong H\rtimes (G/H)$, then we assume the morphism $\kappa=1$.

If all above cases are true, then we have $\varpi$ to be trivial $3$-cocycle, and  a group isomorphism $\widehat{H}\rtimes_{\nu} (G/H)\cong \widehat{H}\rtimes (G/H)$. However, $\widehat{H}\rtimes (G/H)$ might not isomorphic to $G$, as the action of $G/H$ on $\widehat{H}$ is different from the action $\rtimes : H\times G/H\rightarrow H$.

\begin{lem}\label{semiproductandproduct} In any cases
only Hopf algebras $k[(\mathbb{Z}_{i}\rtimes\mathbb{Z}_{j})\times \mathbb{Z}_{k}]$ and $k[(\mathbb{Z}_{l}\rtimes\mathbb{Z}_{m})\times \mathbb{Z}_{n}]$ are in same categorically Morita equivalent classes iff $i=l$, $j=m$, $k=n$, where ${\{i,j,k}\}={\{p,q,r}\}={\{l,m,n}\}$.
\end{lem}
\begin{proof}We can assume the group is $(\mathbb{Z}_{p}\rtimes\mathbb{Z}_{q})\times \mathbb{Z}_{r}$. Then any non-trivial abelian normal subgroup is of prime order or the product of two primes.

For the first case, only $\mathbb{Z}_{p}, \mathbb{Z}_{r}$ are normal subgroups. Because $\mathbb{Z}_{p}\rtimes\mathbb{Z}_{q}$ is the only non-abelian group of order $pq$, and $\mathbb{Z}_r\subseteq Z((\mathbb{Z}_{p}\rtimes\mathbb{Z}_{q})\times \mathbb{Z}_{r})$, therefore using the Sylow's theorem \cite{Ro}, we know  there are $p$ Sylow $q$-subgroups, and the conjugate action acts transitively on the set of Sylow $q$-subgroups, so any subgroup of order $q$ can not be normal.

For the second case, any subgroup of order $pq, qr$ , they can not be abelian normal subgroup. Since $\mathbb{Z}_p, \mathbb{Z}_r$ are the unique subgroups of order $p, r$, any subgroup $N$ of order $pq$ or $qr$, it is normal is equivalent to the normality of subgroup of order $q$, this is impossible, hence $\mathbb{Z}_{p}\times\mathbb{Z}_{r}$  is the only normal subgroup of order $pr$. In any case, the group extension $H\hookrightarrow G\twoheadrightarrow G/H$ is splitting, we have  $\widehat{H}\rtimes (G/H)$ are isomorphic to itself $G$, as for group of order $pqr$, up to isomorphism there is only one non-abelian group of order $pqr$ with center isomorphic to $\mathbb{Z}_r$.
\end{proof}
\begin{pro}\label{onesemiproduct}In cases $(1),(2),(3),(5),(6)$ of [Table \ref{table}], the number of  categorically Morita equivalent classes of semisimple Hopf algebras is equal to the number of isomorphic groups of order $pqr$.
\end{pro}
\begin{proof}The proof is easy: In these cases,  all semsimple Hopf algebras of dimension $pqr$ are trivial Hopf algebras, they all have the form like as:  $k[\mathbb{Z}_p\times\mathbb{Z}_q\times\mathbb{Z}_r]$,  $k[\mathbb{Z}_k\times(\mathbb{Z}_l\rtimes\mathbb{Z}_m)]$, where ${\{k,l,m}\}={\{p,q,r}\}$. By lemma \ref{abelian} , we know $k[\mathbb{Z}_p\times\mathbb{Z}_q\times\mathbb{Z}_r]$ and $k[\mathbb{Z}_k\times(\mathbb{Z}_l\rtimes\mathbb{Z}_m)]$ are not categorically Morita equivalent. Therefore lemma \ref{semiproductandproduct} implies the proof.
\end{proof}
\begin{pro}\label{case4}For case $(4)$, there are $4$ $($if $r=2)$ or $3+\frac{r-1}{2}$ $($if $r>2)$ categorically Morita equivalent classes of semisimple Hopf algebras of dimension $pqr$.
\end{pro}
\begin{proof}In this case, we have $r\mid(q-1)$, $r\mid(p-1)$ and $q\nmid(p-1)$. Semisimple Hopf algebras like $k[\mathbb{Z}_{pqr}]$, $k[(\mathbb{Z}_{p}\rtimes \mathbb{Z}_{r})\times\mathbb{Z}_{q}]$, $k[(\mathbb{Z}_{q}\rtimes \mathbb{Z}_{r})\times\mathbb{Z}_{p}]$ belong different categorically Morita equivalent classes, and by proposition \ref{semiproductandproduct}, we also know that they are not categorically Morita equivalent to other Hopf algebras. So we just need to consider groups of type $G_4$, all the non-trivial abelian normal groups are $\mathbb{Z}_p$, $\mathbb{Z}_q$, $\mathbb{Z}_p\times \mathbb{Z}_q$; if we take $H=\mathbb{Z}_p$, then $\widehat{H}\rtimes(G/H)\cong G$, as  $\widehat{H}\rtimes\mathbb{Z}_r\cong\mathbb{Z}_p\rtimes\mathbb{Z}_r$ is unique up to isomorphism, and $\widehat{H}\rtimes\mathbb{Z}_q\cong\widehat{H}\times\mathbb{Z}_q$. If $H=\mathbb{Z}_q$ or $\mathbb{Z}_p\times \mathbb{Z}_q$, then $\widehat{H}\rtimes (G/H)$ is isomorphic to $\widetilde{G_4}$. Therefore, by the ``duality" of the structures of group $G_4$ and $\widetilde{G_4}$, the proof is finished.
\end{proof}

Then we solve the case $(7)$. We first determine finite group of order $pqr$ satisfying  $qr\mid (p-1)$ and $r\nmid(q-1)$. As listed in [Table \ref{table} ], there are four isomorphic classes of groups, and a direct computation shows the groups listed below are not isomorphic to each other : $\mathbb{Z}_{pqr}\cong\mathbb{Z}_{p}\times\mathbb{Z}_{q}\times \mathbb{Z}_{r}$, $(\mathbb{Z}_{p}\rtimes \mathbb{Z}_{r})\times\mathbb{Z}_{q}$, $(\mathbb{Z}_{p}\rtimes \mathbb{Z}_{q})\times\mathbb{Z}_{r}$, $\mathbb{Z}_{p}\rtimes (\mathbb{Z}_{q}\times\mathbb{Z}_{r})\cong \mathbb{Z}_{p}\rtimes \mathbb{Z}_{qr}$. Here we require all the actions $\rtimes$ are not trivial.

\begin{thm}\label{categoricalM1}There four categorically Morita inequivalent classes in case $(7)$: $k[\mathbb{Z}_{pqr}]$,  $k[(\mathbb{Z}_{p}\rtimes\mathbb{Z}_{r})\times \mathbb{Z}_{q}]$, $k[(\mathbb{Z}_{p}\rtimes\mathbb{Z}_{q})\times \mathbb{Z}_{r}]$ and $\mathcal{A}_p(q;r)$.
\end{thm}
\begin{proof}The semisimple Hopf algebras $\mathcal{A}_p(q;r),~\mathcal{A}_p(r;q)$, $k[\mathbb{Z}_{p}\rtimes(\mathbb{Z}_{q}\times \mathbb{Z}_{r})]$, and $k^{\mathbb{Z}_{p}\rtimes(\mathbb{Z}_{q}\times \mathbb{Z}_{r})}$ are in the same categorical Morita equivalent class, for we have fusion categories equivalences $Rep(\mathcal{A}_p(q;r))\cong \mathcal{C}(\mathbb{Z}_{p}\rtimes(\mathbb{Z}_{q}\times \mathbb{Z}_{r}),1,\mathbb{Z}_{r},1)$ and $Rep(H)^*_{Vec}\cong Rep(H^*)^{op}$ \cite{EGNO} , and also Hopf algebras isomorphism $(\mathcal{A}_p(q;r))^*\cong \mathcal{A}_p(r;q)$ by proposition \ref{DualHopfalgebras}, whence they are in the same class by definition \ref{catgoricalHopfalgebra}.

Combined with lemma \ref{abelian} and lemma \ref{semiproductandproduct}, we complete the proof.
\end{proof}
In the last, we determine the categorically Morita equivalent classes of case $(8)$.
\begin{thm}\label{categoricalM2}There are $6$ $($if $r=2)$ or $5+\frac{r-1}{2}$ $($ if $r>2)$ categorically Morita inequivalent classes in case $(8)$.
\end{thm}
\begin{proof}As in proposition  \ref{onesemiproduct}  and Theorem \ref{categoricalM1}, Hopf algebras $k[\mathbb{Z}_{pqr}]$,  $k[(\mathbb{Z}_{p}\rtimes\mathbb{Z}_{r})\times \mathbb{Z}_{q}]$, $k[(\mathbb{Z}_{p}\rtimes\mathbb{Z}_{q})\times \mathbb{Z}_{r}]$,  $k[(\mathbb{Z}_{q}\rtimes\mathbb{Z}_{r})\times \mathbb{Z}_{p}]$, $\mathcal{A}_p(q;r)$ are in different categorically Morita equivalent classes, and they are not categorically Morita equivalent to any other type of Hopf algebras. To determine categorically Morita equivalent classes of the rest of group algebras is similar to the  proposition \ref{case4} , then we finish the proof.
\end{proof}
\begin{rmk}If $r>3$, then we have $G_4$ and its ``dual" group $\widetilde{G_4}$ are categorically Morita equivalent, but they are not isomorphic. Therefore, by Theorem \ref{monoidalM} , Theorem \ref{categoricalM2} and proposition \ref{case4}, we also obtain examples of finite groups which are categorically Morita equivalent but not isocategorical.
\end{rmk}
\section*{Acknowledgements}
The author thanks his PhD supervisor Professor Naihong Hu for valuable comments and suggestions.

\end{document}